\documentclass[preprint]{elsarticle}
\journal{Complex Analysis and Operator Theory}

\usepackage{enumerate}  
\usepackage{amsmath}    
\usepackage{amsthm}     
\usepackage{mathrsfs}   
\usepackage{amsfonts}
\usepackage{amssymb}
\usepackage{mathtools} 
\usepackage{verbatim}	
\usepackage{color}	
\usepackage{hyperref} 

\usepackage[T1]{fontenc}

\theoremstyle{plain}                
\newtheorem{thm}{Theorem}[section]  
\newtheorem{prop}[thm]{Proposition}

\theoremstyle{definition}
\newtheorem{defn}[thm]{Definition}

\theoremstyle{remark}
\newtheorem{rem}[thm]{Remark}

\DeclareMathOperator{\Arg}{Arg}   
\DeclareMathOperator{\sgn}{sgn}

\providecommand{\abs}[1]{\lvert#1\rvert} 
\providecommand{\norm}[1]{\lVert#1\rVert}
\providecommand{\inner}[1]{\langle#1\rangle} 

\def\re{\mathop{\rm Re}\nolimits}
\def\CC{\mathbb C}
\def\RR{\mathbb R}
\def\NN{\mathbb N}

\def\exp{\mathop{\rm e}\nolimits}

\newcommand{\overbar}[1]{\mkern 1.5mu\overline{\mkern-1.5mu#1\mkern-1.5mu}\mkern 1.5mu}

\begin{document}

\begin{frontmatter}

\title{Admissibility of diagonal state-delayed systems with a one-dimensional input space}
\author[UoL]{Jonathan R. Partington}
\author[SUT]{Rados{\l}aw Zawiski\corref{correspondingauthor}\fnref{additionalaffiliation}}
\ead{radoslaw.zawiski@polsl.pl}
\cortext[correspondingauthor]{Corresponding author}
\fntext[additionalaffiliation]{The research presented here was done when the second author was a Marie Curie Research Fellow at the School of Mathematics, University of Leeds, UK}

\address[UoL]{School of Mathematics, University of Leeds, LS2 9JT Leeds, UK}
\address[SUT]{Institute of Automatic Control, Silesian University of Technology, \\Akademicka 16, 44-100 Gliwice, Poland}

\begin{abstract}
In this paper we investigate admissibility of the control operator $B$ in a Hilbert space state-delayed dynamical system setting of the form $\dot{z}(t)=Az(t-\tau)+Bu(t)$, where $A$ generates a diagonal semigroup and $u$ is a scalar input function. Our approach is based on the Laplace embedding between $L^2$ and the Hardy space. The sufficient conditions for infinite-time admissibility are stated in terms of eigenvalues of the generator and in terms of the control operator itself. 
\end{abstract}


\end{frontmatter}

\section{Introduction}
In this article we analyse dynamical system with delay in the state variable from the perspective of admissibility of the control operator. Thus the object of our interest is an abstract dynamical system
\begin{equation}\label{eqn diagonal delay system introduction}
\left\{\begin{array}{ll}
        \dot{z}(t)=Az(t-\tau)+Bu(t)\\
        z(0)=z_0,        \\
        \end{array}
\right.
\end{equation}
where, in general, $A:D(A)\subset X\rightarrow X$ is the infinitesimal generator of a $C_0$-semigroup $(T(t))_{t\geq0}$ on $X$. The Hilbert space $X$ possesses a sequence of normalized eigenvectors $(\phi_k)_{k\in\NN}$ forming a Riesz basis, with associated eigenvalues $(\lambda_k)_{k\in\NN}$. 
The input function is $u\in L^2(0,\infty;\CC)$, $B$ is the control operator and  $0<\tau<\infty$ is a delay.

Infinite-time admissibility of $B$ in the undelayed case of \eqref{eqn diagonal delay system introduction} is well analysed and the necessary and sufficient conditions for it were given using e.g. Carleson measures. 
In particular, the link between Carleson measures and infinite-time admissibility was studied in \cite{Ho_Russell_1983,Ho_Russell_1983_err,Weiss_1988}. 
Those results were extended to normal semigroups \cite{Weiss_1999}, then generalized to the case when $u\in L^2(0,\infty;t^{\alpha}dt)$ for $\alpha\in(-1,0)$ in \cite{Wynn_2010} and further to the case $u\in L^2(0,\infty;w(t)dt)$ in \cite{Jacob_Partington_Pott_2013,Jacob_Partington_Pott_2014}. 
For a thorough presentation of admissibility results, not restricted to diagonal systems, for the undelayed case we refer the reader to \cite{Jacob_Partington_2004} and a rich list of references therein. 

The results in \cite{Grabowski_Callier_1996} and \cite{Engel_1999} form a basis for considerations in \cite{Batkai_Piazzera} in terms of developing a correct setting in which we conduct the admissibility analysis for state-delayed diagonal systems. 
The same setting is used by us for the admissibility analysis in a more general case when \eqref{eqn diagonal delay system introduction} takes a form of the so-called retarded equation, where we assume only a contraction property of the undelayed semigroup $(T(t))_{t\geq0}$ (full details will be published elsewhere \cite{Partington_Zawiski_2018a}) 

Section~\ref{sec:2} contains the necessary background results, leading to the main results in Section~\ref{sec:3}.
An example is given in Section~\ref{sec:4}, and some conclusions are given in Section~\ref{sec:5}.

\section{Preliminaries}\label{sec:2}

Apart from definitions introduced in the previous section throughout this paper we use the following Sobolev spaces (see \cite{Kreuter_2015} for vector valued functions or \cite[Chapter 5]{Evans} for functionals):
$W^{1,2}(J,X):=\{f\in L^2(J,X):\frac{d}{dt}f(t)\in L^2(J,X)\}$, $W_{c}^{1,2}(J,X):=\{f\in W^{1,2}(J,X):f \hbox{ has compact support}\}$ and $W_{0}^{1,2}(J,X):=\{f\in W^{1,2}(J,X):\ f(\partial J)=0\}$, where $J$ is an interval. 

For any $\alpha\in\mathbb{R}$ we denote $\CC_{\alpha}:=\{s\in\mathbb{C}:\text{Re}s>\alpha \}$ with an exception for two special cases, namely $\CC_{+}:=\{s\in\mathbb{C}:\re s>0 \}$ and $\CC_{-}:=\{s\in\mathbb{C}:\text{Re}s<0\}$.

The Hardy space $H^2(\CC_+)$ consists of all analytic functions $f:\CC_+\rightarrow\CC$ for which 
\begin{equation}\label{eqn defn Hardy space element}
\sup_{\alpha>0}\int_{-\infty}^{\infty}\abs{f(\alpha+i\omega)}^2 \, d\omega<\infty.
\end{equation}
If $f\in H^2(\CC_+)$ then for almost every $\omega\in\RR$ the limit
\begin{equation}\label{eqn defn boundary trace}
f^*(i\omega)=\lim_{\alpha\downarrow 0}f(\alpha+i\omega)
\end{equation}
exists and defines a function  $f^*\in L^2(i\RR)$ called the \textit{boundary trace} of $f$.  Using boundary traces $H^2(\CC_+)$ is made into a Hilbert space with the inner product defined as
\begin{equation}\label{eqn defn inner product on H2 space}
\inner{f,g}_{H^2(\CC_+)}:=\inner{f^*,g^*}_{L^2(i\RR)}:=\frac{1}{2\pi}\int_{-\infty}^{+\infty}f^*(i\omega)\bar{g}^*(i\omega) \, d\omega
\end{equation}
for every $f,g\in H^2(\CC_+)$. For more information about Hardy spaces see \cite{Partington_1988}, \cite{Garnett} or \cite{Koosis}. We also make use of the following 
\begin{thm}[Paley-Wiener]\label{thm Paley-Wiener}
Let $Y$ be a Hilbert space. Then the Laplace transform $\mathcal{L}:L^2(0,\infty;Y)\rightarrow H^2(\CC_+;Y)$ is an isometric isomorphism.
\end{thm}
For a detailed proof of Theorem~\ref{thm Paley-Wiener} see \cite[Chapter 19]{Rudin_1987} for the scalar version or \cite[Theorem 1.8.3]{Arendt_et_al} for the vector-valued one.

\subsection{The delayed equation setting}

For details of the setting in which we consider a state-delayed diagonal system see \cite[Chapter VI.6]{Engel_Nagel} and \cite[Chapter 3.1]{Batkai_Piazzera}. Consider a function  $z:[-\tau,\infty)\rightarrow X$. For each $t\geq0$ we call the function $z_{t}:[-\tau,0]\rightarrow X$, $z_{t}(\sigma):=z(t+\sigma)$, a \textit{history segment} with respect to $t\geq0$. With history segments we consider a function called the \textit{history function} of $z$, that is $h_{z}:[0,\infty)\rightarrow L^2(-\tau,0;X)$, $h_z(t):=z_{t}$. In \cite[Lemma 3.4]{Batkai_Piazzera} we find the following

\begin{prop}\label{prop history function derivative}
Let $1\leq p<\infty$ and $z:[-\tau,\infty)\rightarrow X$ be a function which belongs to $W_{loc}^{1,p}(-\tau,\infty;X)$. Then the history function $h_z:t\rightarrow z_t$ of $z$ is continuously differentiable from $\mathbb{R}_+$ into $L^p(-\tau,0;X)$ with derivative 
\[
\frac{\partial}{\partial t}h_{z}(t)=\frac{\partial}{\partial\sigma}z_{t}.
\]
\end{prop}

Define the Cartesian product $\mathcal{X}:=X\times L^2(-\tau,0;X)$ with an inner product 
\begin{equation}\label{eqn defn of inner product on Cartesian product}
 \bigg\langle\binom{x}{f},\binom{y}{g}\bigg\rangle_{\mathcal{X}}:=\langle x,y\rangle_{X}+\langle f,g\rangle_{L^2(-\tau,0;\CC)}.
\end{equation}
Then $\mathcal{X}$ becomes a Hilbert space $(\mathcal{X},\|\cdot\|_{\mathcal{X}})$ with the norm $\|\binom{x}{f}\|_{\mathcal{X}}^2=\|x\|_{X}^{2}+\|f\|_{L^2}^{2}$. Consider a linear, autonomous delay differential equation of the form
\begin{equation}\label{eqn delay autonomous diff eq}
\left\{\begin{array}{ll}
        \dot{z}(t)=Az(t)+\Psi z_t\\
        z(0)=x,        \\
        z_0=f,		\\
       \end{array}
\right.
\end{equation}
where $\Psi\in\mathcal{L}(W^{1,2}(-\tau,0;X),X)$ is a \textit{delay operator}, the pair $x\in D(A)$ and $f\in L^2(-\tau,0;X)$ forms an initial condition. Due to Proposition~\ref{prop history function derivative} equation \eqref{eqn delay autonomous diff eq} may be written as an abstract Cauchy problem
\begin{equation}\label{eqn defn abstract Cauchy problem}
\left\{\begin{array}{ll}
        \dot{v}(t)=\mathcal{A}v(t)\\
        v(0)=\binom{x}{f},        \\
       \end{array}
\right.
\end{equation} 
where $v:t\rightarrow\binom{z(t)}{z_t}\in\mathcal{X}$ 
and $\mathcal{A}$ is an operator on $\mathcal{X}$ defined as
\begin{equation}\label{eqn defn abstract A}
\mathcal{A}:=\left(\begin{array}{cc} 
								A & \Psi \\ 
								0 & \frac{d}{d\sigma}	
						\end{array}\right),
\end{equation}
with domain
\begin{equation}\label{eqn defn abstract A domain}
D(\mathcal{A}):=\bigg\{\binom{x}{f}\in D(A)\times W^{1,2}(-\tau,0;X):\ f(0)=x\bigg\}.
\end{equation}
The operator $(\mathcal{A},D(\mathcal{A}))$ is closed and densely defined on $\mathcal{X}$ \cite[Lemma 3.6]{Batkai_Piazzera}. Let $\mathcal{A}=\mathcal{A}_0+\mathcal{A}_{\Psi}$, where 
\begin{equation}\label{eqn defn A_0}
\mathcal{A}_0:=\left(\begin{array}{cc} 
								A & 0 \\ 
								0 & \frac{d}{d\sigma}	
						\end{array}\right),\qquad D(\mathcal{A}_0)=D(\mathcal{A}),
\end{equation}
and
\begin{equation}\label{eqn defn  A_psi}
\mathcal{A}_{\Psi}:=\left(\begin{array}{cc} 
								0 & \Psi \\ 
								0 & 0	
						\end{array}\right)\in\mathcal{L}\big(X\times W^{1,2}(-\tau,0;X),\mathcal{X}\big).
\end{equation}
%

%
We will need the following for the Miyadera--Voigt Perturbation Theorem and a description of admissibility.

\begin{defn}\label{defn Sobolev tower}
Let $\beta\in\rho(A)$ and denote $(X_1,\norm{\cdot}_1):=(D(A),\norm{\cdot}_1)$ with $\norm{\cdot}_1:=\norm{(\beta I-A)x}\ (x\in D(A))$ .
 
Similarly, we set $\norm{x}_{-1}:=\norm{(\beta I-A)^{-1}x}\ (x\in X)$. Then the space $(X_{-1},\norm{\cdot}_{-1})$ denotes the completion of $X$ under the norm $\norm{\cdot}_{-1}$. For $t\geq0$ we define $T_{-1}(t)$ as the continuous extension of $T(t)$ to the space  $(X_{-1},\norm{\cdot}_{-1})$. 
\end{defn}
In the sequel, much of our reasoning is justified by the following Proposition, to which we do not refer directly but 
include here for the reader's convenience.
\begin{prop}\label{prop Hilbert rigged space}
With notation of Definition \ref{defn Sobolev tower} we have the following
\begin{itemize}
 \item[(i)] The spaces $(X_1,\norm{\cdot}_1)$ and $(X_{-1},\norm{\cdot}_{-1})$ are independent of the choice of $\beta\in\rho(A)$.
 \item[(ii)] $(T_1(t))_{t\geq0}$ is a strongly continuous semigroup on the Banach space \\$(X_1,\norm{\cdot}_1)$ and we have $\norm{T_1(t)}_1=\norm{T(t)}$ for all $t\geq0$.
 \item[(iii)] $(T_{-1}(t))_{t\geq0}$ is a strongly continuous semigroup on the Banach space \\$(X_{-1},\norm{\cdot}_{-1})$ and we have $\norm{T_{-1}(t)}_{-1}=\norm{T(t)}$ for all $t\geq0$.
\end{itemize}
\end{prop}
See \cite[Chapter II.5]{Engel_Nagel} or \cite[Chapter 2.10]{Tucsnak_Weiss} for more details on these elements. A sufficient condition for $P\in\mathcal{L}(X_1,X)$ to be a perturbation of Miyadera-Voigt class, and hence implying that $A+P$ is a generator on $X$, takes the form of  \cite[Corollary III.3.16]{Engel_Nagel} 

\begin{prop}\label{prop Miyadera-Voigt sufficient condition of well posedness}
Let $(A,D(A))$ be the generator of a strongly continuous semigroup $\big(T(t)\big)_{t\geq0}$ on a Banach space $X$ and let ${P\in\mathcal{L}(X_1,X)}$ be a perturbation which satisfies
\begin{equation}\label{eqn condition on perturbation for well-posedness}
\int_{0}^{t_0}\norm{PT(r)x} \, dr\leq q\Vert x\Vert\qquad\forall x\in D(A)
\end{equation}
for some $t_0>0$ and $0\leq q<1$. Then the sum $A+P$ with domain $D(A+P):=D(A)$ generates a strongly continuous semigroup $(S(t))_{t\geq0}$ on $X$. 
\end{prop}

To describe the resolvent of $(\mathcal{A},D(\mathcal{A}))$, let us introduce the notation
\[
A_{0}:=\frac{d}{d\sigma},\qquad D(A_{0})=\{z\in W^{1,2}(-\tau,0;X):z(0)=0\},
\]
for the generator of the nilpotent left shift semigroup on $L^p(-\tau,0;X)$. For $s\in\mathbb{C}$ define $\epsilon_{s}:[-\tau,0]\rightarrow\mathbb{C}$, $\epsilon_{s}(\sigma):=e^{s \sigma}$. Define also $\Psi_{s}\in\mathcal{L}(D(A),X)$, $\Psi_{s}x:=\Psi(\epsilon_{s}(\cdot)x)$. Then \cite[Proposition 3.19]{Batkai_Piazzera} provides
\begin{prop}\label{prop abstract A resolvent operator}
For $s\in\mathbb{C}$ and for all $1\leq p<\infty$ we have 
\[
s\in\rho(\mathcal{A}) \quad \hbox{if and only if} \quad s\in\rho(A+\Psi_{s}).
\]
Moreover, for $s\in\rho(\mathcal{A})$ the resolvent operator $R(s,\mathcal{A})$ is given by
\begin{equation}\label{eqn relosvent of abstract A}
R(s,\mathcal{A})=\left(\begin{array}{ll} 
	R(s,A+\Psi_{s}) & R(s,A+\Psi_{s})\Psi R(s,A_0) \\ 
	\epsilon_{s}R(s,A+\Psi_{s}) & (\epsilon_{s}R(s,A+\Psi_{s})\Psi+I)	R(s,A_0)
								  \end{array}\right).
\end{equation}

\end{prop}

\subsection{The admissibility problem}
The basic object in the formulation of admissibility problem is a linear system and its mild solution
\begin{equation}\label{eqn basic object}
 \frac{d}{dt}x(t)=Ax(t)+Bu(t);\quad x(t)=T(t)x_0+\int_{0}^{t}T(t-s)Bu(s) \, ds,
\end{equation}
where $x:[0,\infty)\rightarrow X$, $u\in V$ where $V$ is a space of measurable functions from $[0,\infty)$ to $U$ and $B$ is a \textit{control operator}; $x_0\in X$ is an initial state.

In many practical examples the control operator $B$ is unbounded, hence \eqref{eqn basic object} is viewed on an extrapolation space $X_{-1}\supset X$ where $B\in\mathcal{L}(U,X_{-1})$. To ensure that the state $x(t)$ lies in $X$ it is sufficient that $\int_{0}^{t}T_{-1}(t-s)Bu(s) \, ds\in X$ for all inputs $u\in V$. Put differently, we have 

\begin{defn}\label{defn finite- and infinite-time admissibility}
 The control operator $B\in\mathcal{L}(U,X_{-1})$ is said to be \textit{finite-time admissible} for a semigroup $\big(T(t)\big)_{t\geq0}$ on a Hilbert space $X$ if for each $\tau>0$ there is a constant $c(\tau)$ such that the condition 

\begin{equation}\label{eqn defn finite-time admissibility by norm inequality}
 \Big\|\int_{0}^{\tau}T_{-1}(\tau-s)Bu(s) \, ds\Big\|_X\leq c(\tau)\norm{u}_V
\end{equation}

holds for all inputs $u$, and an \textit{infinite-time admissible} if the condition \eqref{eqn defn finite-time admissibility by norm inequality} holds for all $\tau>0$ with $c(\tau)$ uniformly bounded.
\end{defn}

In the sequel, we denote the restriction (extension) of $T(t)$ described in Definition~\ref{defn Sobolev tower} by the same symbol $T(t)$, since this is unlikely to lead to confusions.

\section{Diagonal non-autonomous delay systems}\label{sec:3}
We begin with an analysis of \eqref{eqn diagonal delay system introduction} in a more concrete setting. Consider 
the system 
\begin{equation}\label{eqn diagonal non-autonomous system}
\left\{\begin{array}{ll}
        \dot{z}(t)=Az(t-\tau)+Bu(t)\\
        z(0)=x,        \\
        z_0=f,		\\
       \end{array}
\right.
\end{equation}
where the state space is $X:=l^2(\CC)$, the control function $u\in L^2(0,\infty;\CC)$ and $(\lambda_k)_{k\in\NN}$ is a sequence in $\CC$ such that 
\begin{equation}\label{eqn sequence of eigenvalues of the generator}
\lambda_k\in\CC_{-}\quad\forall k\in\NN.
\end{equation}
The semigroup generator $(A,D(A))$ is defined by 
\begin{equation}\label{eqn defn diagonal generator}
(Az)_k:=\lambda_{k}z_{k},\quad D(A):=\bigg\{z\in l^2(\CC):\sum_{k\in\NN}(1+|\lambda_k|^2)|z_k|^2<\infty\bigg\}.
\end{equation}
%

As the space $X_1$ we take  $(D(A),\norm{\cdot}_{gr})$, where the graph norm is equivalent to
\[
\norm{z}_1^2=\sum_{k\in\NN}(1+|\lambda_k|^2)|z_k|^2.
\]
The adjoint generator $A^*$ is represented in the same way, with the sequence $(\bar{\lambda}_k)_{k\in\NN}$ in place of $(\lambda_k)_{k\in\NN}$. This gives $D(A^*)=D(A)$. The space $X_{-1}$ consists of all sequences $z=(z_k)_{k\in\NN}\in\CC^\NN$ for which
\[
\sum_{k\in\NN}\frac{|z_k|^2}{1+|\lambda_k|^2}<\infty,
\]
and the square root of the above series gives an equivalent norm on $X_{-1}$. The space  $X_{-1}$ is the same as  $X_{-1}^d$, where the latter one is the equivalent of $X_{-1}$ should the construction in Definition~\ref{defn Sobolev tower} be based on $A^*$ instead of $A$.

Note also that the operator $B\in\mathcal{L}(\CC,X_{-1})$ is represented by the sequence $(b_k)_{k\in\NN}\in\CC^\NN$, as $\mathcal{L}(\CC,X_{-1})$ can be identified with $X_{-1}$. 

The above is   the standard setting for diagonal systems; we refer the reader to \cite[Chapters 2.6 and 5.3]{Tucsnak_Weiss} for more details.

\begin{rem}
Although we restrict ourselves to contraction semigroups, this does not lead to loss of generality due to the semigroup rescaling property. That is when $A$ does not generate a contraction semigroup, we may replace it with a shifted version $A-\alpha I$ for a sufficiently large $\alpha>0$. This does not change the admissibility of control operator for the rescaled semigroup, but may change the infinite time admissibility.
\end{rem}

\subsection{Analysis of a single component}
Let us now focus on the $k$-th component of \eqref{eqn diagonal non-autonomous system}, that is 
\begin{equation}\label{eqn k-th component of diagonal non-autonomous system}
\left\{\begin{array}{ll}
        \dot{z}_k(t)=\lambda_{k}z_{k}(t-\tau)+b_ku(t)\\
        z_k(0)=x_k,        \\
        z_{0_k}=f_k,		\\
       \end{array}
\right.
\end{equation}
where $\lambda_k,b_k,x_k\in\CC$, $f_k:=\langle f,l_k\rangle_{L^2(-\tau,0;X)}l_k$ with $l_k$ being the $k$-th component of the standard orthonormal basis in $L^2(-\tau,0;X)$. 

For the sake of clarity of notation, let us now until the end of this subsection drop the subscript $k$ and rewrite \eqref{eqn k-th component of diagonal non-autonomous system} in the form
\begin{equation}\label{eqn k-th component Cauchy problem}
\left\{\begin{array}{ll}
        \dot{z}(t)=\Psi z_t+bu(t)\\
        z(0)=x,        \\
        z_{0}=f,		\\
       \end{array}
\right.
\end{equation}
where the delay operator $\Psi\in\mathcal{L}(W^{1,2}(-\tau,0;\CC),\CC)$ is defined as
\begin{equation}\label{eqn delay operator example}
\Psi(f):=\lambda f(-\tau)\qquad\forall f\in W^{1,2}(-\tau,0;\CC).
\end{equation}
Observe that, without the input function $bu\in L^2(0,\infty;\CC)$, system \eqref{eqn k-th component Cauchy problem} is a simplified form of \eqref{eqn delay autonomous diff eq}. As for such, we can apply the procedure described in the Preliminaries section and represent it as an abstract Cauchy problem of the form \eqref{eqn defn abstract Cauchy problem}. For that purpose note that 
\begin{equation}\label{eqn defn k-th component Cartesian product}
\mathcal{X}:=\CC\times L^2(-\tau,0;\CC) 
\end{equation}
with an inner product 
\begin{equation}\label{eqn defn k-th component inner product Cartesian product}
 \bigg\langle\binom{x}{f},\binom{y}{g}\bigg\rangle_{\mathcal{X}}:=x\bar{y}+\langle f,g\rangle_{L^2(-\tau,0;\CC)}\quad\forall \binom{x}{f},\binom{y}{g}\in\mathcal{X}.
\end{equation}
What follows is the non-autonomous Cauchy problem describing the dynamics of the $k$-th component

\begin{equation}\label{eqn k-th component abstract Cauchy problem}
\left\{\begin{array}{ll}
        \dot{v}(t)=\mathcal{A}v(t)+\mathcal{B}u(t)\\
        v(0)=\binom{x}{f},        \\
       \end{array}
\right.
\end{equation} 
where $v:t\rightarrow\binom{z(t)}{z_t}\in\mathcal{X}$ and $\mathcal{A}$ is an operator on $\mathcal{X}$ defined as
\begin{equation}\label{eqn defn k-th component abstract A}
\mathcal{A}:=\left(\begin{array}{cc} 
								0 & \Psi \\ 
								0 & \frac{d}{d\sigma}	
						\end{array}\right),
\end{equation}
with domain
\begin{equation}\label{eqn defn k-th component abstract A domain}
D(\mathcal{A}):=\bigg\{\binom{x}{f}\in \CC\times W^{1,2}(-\tau,0;\CC):\ f(0)=x\bigg\},
\end{equation}
and $\mathcal{B}:=\binom{b}{0}\in\mathcal{L}(\CC,\mathcal{X}_{-1})$.To state explicitly how the $\mathcal{X}_{-1}$ space looks like we use again \eqref{eqn defn k-th component abstract A} and \eqref{eqn defn k-th component abstract A domain}  as well as Proposition 3.1 from \cite{Partington_Zawiski_2018a}. As a  result,
\begin{equation}\label{eqn defn k-th component extended space}
\mathcal{X}_{-1}=\CC\times W^{-1,2}(-\tau,0;\CC),
\end{equation}
where $W^{-1,2}(-\tau,0;\CC)$ is the dual to $W_{0}^{1,2}(-\tau,0;\CC)$ with respect to the pivot space $L^2(-\tau,0;\CC)$. The generator $\mathcal{A}$ may again be represented as $\mathcal{A}=\mathcal{A}_0+\mathcal{A}_{\Psi}$, where 
\begin{equation}\label{eqn defn k-th component A_0}
\mathcal{A}_0:=\left(\begin{array}{cc} 
								0 & 0 \\ 
								0 & \frac{d}{d\sigma}	
						\end{array}\right),\quad D(\mathcal{A}_0)=D(\mathcal{A}),
\end{equation}
and
\begin{equation}\label{eqn defn k-th component A_psi}
\mathcal{A}_{\Psi}:=\left(\begin{array}{cc} 
								0 & \Psi \\ 
								0 & 0	
						\end{array}\right)\in\mathcal{L}\big(\CC\times W^{1,2}(-\tau,0;\CC),\mathcal{X}\big).
\end{equation}
We have the following
\begin{prop}\label{prop wellposedness of k-th component ACP}
The abstract Cauchy problem \eqref{eqn k-th component abstract Cauchy problem} is well-posed.
\end{prop}
\begin{proof}
The delay operator $\Psi$ defined in \eqref{eqn delay operator example} is an example of a much wider class of delay operators, with which condition \eqref{eqn condition on perturbation for well-posedness} is satisfied and $(\mathcal{A},D(\mathcal{A}))$ in \eqref{eqn defn k-th component abstract A} remains a generator of a strongly continuous semigroup $(\mathcal{T}(t))_{t\geq0}$ on $\mathcal{X}$. See \cite[Chapter 3.3 and Example 3.28]{Batkai_Piazzera} for details.
\end{proof}

Due to Proposition~\ref{prop wellposedness of k-th component ACP} we can formally write the $\mathcal{X}_{-1}$-valued $k$-th component mild solution of \eqref{eqn k-th component abstract Cauchy problem}

\begin{equation}\label{eqn k-th component abstract Cauchy problem mild solution}
 v(t)=\mathcal{T}(t)v(0)+\int_{0}^{t}\mathcal{T}(t-s)\mathcal{B}u(s) \, ds,
\end{equation}
where $\mathcal{T}(t)\in\mathcal{L}(\mathcal{X}_{-1})$ and the control operator is again $\mathcal{B}=\binom{b}{0}\in\mathcal{L}(\CC,\mathcal{X}_{-1})$. The following Proposition gives information concerning spectral properties and the resolvent operator $R(s,\mathcal{A})$.
\begin{prop}\label{prop k-th component abstract A resolvent operator}
For $s\in\mathbb{C}$ and for all $1\leq p<\infty$ there is 
\begin{equation}\label{k-th component condition on resolvent sets}
s\in\rho(\mathcal{A})\text{ if and only if } s\in\rho(\Psi_{s}).
\end{equation}
Moreover, for $s\in\rho(\mathcal{A})$ the resolvent operator $R(s,\mathcal{A})$ is given by

	\begin{equation}\label{eqn k-th component resolvent of abstract A}
	R(s,\mathcal{A})=\left(\begin{array}{ll} 
		R(s,\Psi_{s}) & R(s,\Psi_{s})\Psi R(s,A_0) \\ 
		\epsilon_{s}R(s,\Psi_{s}) & (\epsilon_{s}R(s,\Psi_{s})\Psi+I)R(s,A_0)
									  \end{array}\right),
	\end{equation}
where $R(s,\Psi_s)\in\mathcal{L}(\CC)$,
\begin{equation}
 R(s,\Psi_s)=\frac{1}{s-\lambda \exp^{-s\tau}} \quad \forall s\in\CC_{\abs{\lambda}}
\end{equation}
and $R(s,A_0)\in\mathcal{L}(L^2(-\tau,0;\CC))$,
 \begin{equation}
R(s,A_0)f(r)=\int_{r}^{0}\exp^{s(r-t)}f(t) \,dt
 \quad r\in[-\tau,0]\quad \forall s\in\CC_{\abs{\lambda}}.
\end{equation}

\end{prop}
\begin{proof}
\begin{itemize}
	\item[1.] Condition \eqref{k-th component condition on resolvent sets} and the form of $R(s,\mathcal{A})$ in \eqref{eqn k-th component resolvent of abstract A} follow directly from Proposition~\ref{prop abstract A resolvent operator} and the form of $\mathcal{A}$ given in \eqref{eqn defn k-th component abstract A}.
	
	\item[2.] As is well known, for any Banach space $X$ and operator $A\in\mathcal{L}(X)$ the condition $s\in\sigma(A)$ implies $\abs{s}\leq\norm{A}$. 

\item[3.]
According to the definitions given before Proposition~\ref{prop abstract A resolvent operator} in this case there is $\Psi_s\in\mathcal{L}(\CC)$, $\Psi_{s}x:=\lambda\exp^{-s\tau}x$ and $\norm{\Psi_s}=\abs{\lambda}\exp^{-\re s\tau}$.

The equation $(\mu-\Psi_s )x=y$ has a unique solution $x \in \CC$ for each $y \in \CC$
if and only if $\mu \ne \lambda \exp^{-s\tau}$.
Thus $\sigma(\Psi_s)= \{\lambda \exp^{-s\tau} \}$, and so
\[
\{s \in \CC: s \in \sigma(\Psi_s) \}\subset 
\{s \in \CC: |s| \le |\lambda| \exp^{-\re s\tau} \} \subset
\{s \in \CC: \re s \le |\lambda| \}.
\]
Moreover, for $s\neq\lambda\exp^{-s\tau}$ there is
\[
R(s,\Psi_s)	=\frac{1}{s-\lambda\exp^{-s\tau}}.
\]
	
	
	\item[4.] To complete the description of $R(s,\mathcal{A})$ consider now $f\in L^2(-\tau,0;\CC)$,  $g\in W^{1,2}(-\tau,0;\CC)$ and a formal differential equation 
	\begin{equation}\label{eqn diff_eq for resolvent}
	(sI-A_0)g(r)=sg(r)-g'(r)=f(r)
	\end{equation}
	with an initial condition imposed on $f$ in the form 
 $f(0)=0$.
	Solving firstly a homogeneous equation and then using the method of   variation of constants one obtains
	\[
	g(r)=\int_{r}^{0}\exp^{s(r-t)}f(t) \, dt
\quad \forall r\in[-\tau,0]\ \forall s\in\CC_{\abs{\lambda}}
	\]
(see also \cite[p. 174, (6.6)]{kato}).

	Denote now $R_{s}f(r):=g(r)$, $r\in[-\tau,0]$. Then, for every $f\in L^2(-\tau,0;\CC)$ there is $(sI-A_0)R_{s}f(r)=f(r)$ and $R_s:L^2(-\tau,0;\CC)\rightarrow D(A_0)$ and $R_s\in\mathcal{L}(L^2(-\tau,0;\CC))$. Let now $f\in D(A_0)$. A simple check shows that $R_s[(sI-A_0)f(r)]=f(r)$, $r\in[-\tau,0]$. This means that $R_s$ is in fact a resolvent operator and we may write	
	\[
	R(s,A_0)f(r)=R_{s}f(r)=\int_{r}^{0}\exp^{s(r-t)}f(t) \, dt
\quad \forall f\in L^2(-\tau,0;\CC).
	\]
	
\end{itemize}
\end{proof}


Proposition~\ref{prop k-th component abstract A resolvent operator} gives the form of the resolvent $R(s,\Psi_s)$ and assures that it is analytic on $\CC_{\abs{\lambda}}$. The value of $\lambda$ is valid for the given mode only and at this stage $\abs{\lambda}\rightarrow\infty$ is allowed. Thus, as we will later require analyticity of $R(s,\Psi_s)$ in $\CC_{+}$, a different approach is needed. For that reason we turn our attention to the complex coefficient exponential polynomial $P:\CC\rightarrow\CC$,

\begin{equation}\label{eqn defn complex exponential polynomial}
P(s):=s-\lambda\exp^{-s\tau}, 
\end{equation}
where $\lambda\in\CC_{-}$ is a complex coefficient and $\tau>0$. 

The polynomial \eqref{eqn defn complex exponential polynomial} in a more general form $A(s)+B(s)\exp^{-s\tau}$ is known and widely studied in the theory of stability of finite dimensional dynamical systems - see e.g. \cite[Chapter 13]{Bellman_Cooke} or \cite[Chapter 6]{Partington_2004} and references therein. The main difficulty in our case, in comparison to the references given above, is that the coefficients are complex. Nevertheless, we can use a modified Walton--Marshall approach \cite{Walton_Marshall_1987} (or \cite[Proposition 6.2.3]{Partington_2004}), as the following Proposition shows.  

\begin{rem}
We take the principal argument of $\lambda$ to be $\Arg(\lambda)\in(-\pi,\pi]$.
\end{rem}

We shall require the following subset of the complex plane, depending on $\tau>0$:
\begin{equation}\label{eq:lambdatau}
\Lambda_\tau:=\bigg\{\lambda\in\CC_{-}:\Arg(\lambda)\in\big(-\pi,-\frac{\pi}{2}\big)\cup\big(\frac{\pi}{2},\pi\big],\ \abs{\lambda}<\frac{1}{\tau}\big(\abs{\Arg(\lambda)}-\frac{\pi}{2}\big)\bigg\}.
\end{equation}

\begin{prop}\label{prop stability of complex polynomial}For a given $\tau>0$ and $\lambda\in\CC_{-}$ the   condition $\lambda \in \Lambda_\tau$
 is sufficient for the polynomial $P$ defined in \eqref{eqn defn complex exponential polynomial} no to have right half-plane zeros (to be stable). In other words, all the solutions of the characteristic equation $P(s)=0$
belong to $\CC_{-}$.
\end{prop}
\begin{proof}
\begin{itemize}
  \item[1.] Consider initially the case when $\tau=0$. The polynomial $P$ has one root $s_0=\lambda$ and $s_0\in\CC_{-}$. 
  
  \item[2.] Using Rouch{\'e}'s theorem (see e.g. \cite[Theorem 12.2]{Bellman_Cooke}) one can show that the zeros move continuously with $\tau$. As they start in $\CC_{-}$ it remains to establish when they cross the imaginary axis.
  
  \item[3.] 
  At the crossing of the imaginary axis there is $s=i\omega$ for some $\omega\in\RR$ and the characteristic equation takes the form
  \begin{equation}\label{eqn char. eqn. 1}
  s-\lambda\exp^{-s\tau}=0.
  \end{equation}
	By point 2. we can treat \eqref{eqn char. eqn. 1} as an implicit function with $s=s(\tau)$ and check the direction in 	which zeros of it cross the imaginary axis  by analysing the $\sgn\re\frac{ds}{d\tau}$ at $s=i\omega$. By calculating 	the implicit function derivative we obtain
	\[
	\frac{ds}{d\tau}=-\frac{s^2}{1+s\tau}.
	\]
  As $s$ is purely imaginary and $\sgn\re z=\sgn\re z^{-1}$ we have  
  	\[
  	\sgn\re\frac{ds}{d\tau}>0
	\]
	and the zeros cross from the left to the right half-plane. What remains is to find for what $\tau$ this happens.
	
  \item[4.] Taking the complex conjugate of \eqref{eqn char. eqn. 1}  we obtain
  \[
  -s-\bar{\lambda}\exp^{s\tau}=0.
  \]
  Using both of the above equations to eliminate the exponential part we obtain $s^2=-\abs{\lambda}^2$, hence $s=\pm i\abs{\lambda}$. Choosing to work further with $s=i\abs{\lambda}$ and substituting it into \eqref{eqn char. eqn. 1}  we get
  \begin{equation}\label{eqn step_1}
  -i\frac{\lambda}{\abs{\lambda}}=\exp^{i\abs{\lambda}\tau}.
  \end{equation}
  
The corresponding equation for $s=-i|\lambda|$ is
\[
 - i\frac{\abs{\lambda}}{\lambda} =\exp^{i\abs{\lambda}\tau},
\]
which has the same form as \eqref{eqn step_1}, but replacing $\lambda$ by $\bar\lambda$.

  \item[5.] Let now $\lambda=\abs{\lambda}\exp^{i\Arg\lambda}$ where $\Arg(\lambda)\in(-\pi,-\frac{\pi}{2})\cup(\frac{\pi}{2},\pi]$. This gives
  \begin{equation}\label{eqn principal arument for step_1}
  \Arg\lambda-\frac{\pi}{2}\in(-\frac{3\pi}{2},-\pi)\cup(0,\frac{\pi}{2}]
  \end{equation}
  and from \eqref{eqn step_1} we have
  
  \begin{equation}\label{eqn lambda_tau condition}
   0<\abs{\lambda}\tau=\Arg\lambda-\frac{\pi}{2}\leq\frac{\pi}{2}.
  \end{equation}
  The above brings us to an observation that if there exist $\lambda\in\CC_-$ and $\tau>0$ such that $s=i\abs{\lambda}$ is a solution to \eqref{eqn char. eqn. 1} i.e. $s=\lambda\exp^{-s\tau}$ then $\frac{\pi}{2}<\Arg(\lambda)\leq\pi$ and \eqref{eqn lambda_tau condition} is satisfied.
  
  If we choose to work in point 4. with $s=-i\abs{\lambda}$ instead, then by symmetry we obtain that if there exist $\lambda\in\CC_-$ and $\tau>0$ such that $s=-i\abs{\lambda}$ is a solution to $s=\lambda\exp^{-s\tau}$ then $-\pi<\Arg(\lambda)<\frac{\pi}{2}$ and the equation $\abs{\lambda}\tau=-\Arg(\lambda)-\frac{\pi}{2}$ is satisfied.
  
  \item[6.] From the discussion in point 5. we draw two conclusions: 
  \begin{itemize}
    \item[(a)] given a diagonal system, with fixed $(\lambda_k)_{k\in\NN}$, the delay $\tau$ assuring that each mode is stable satisfies
    
    \begin{equation}\label{eqn delay condition on mode stability}
    \tau<\frac{1}{\abs{\lambda_k}}\bigg(\abs{\Arg\lambda_k}-\frac{\pi}{2}\bigg)\qquad\forall k\in\NN,
    \end{equation}
 
    \item[(b)] given a delay $\tau$, the distribution of $(\lambda_k)_{k\in\NN}$ for each mode to remain stable is
    
    \begin{equation}\label{eqn lambda_k condition on mode stability}
    \abs{\lambda_k}<\frac{1}{\tau}\bigg(\abs{\Arg\lambda_k}-\frac{\pi}{2}\bigg)\qquad\forall k\in\NN.
    \end{equation}
    
   \end{itemize}
  Clearly $(\lambda_k)_{k\in\NN}\subset\CC_-$.
\end{itemize}
\end{proof}

In geometrical terms Proposition \ref{prop stability of complex polynomial} states that the stability of $P$ is preserved for given $\tau$ provided that we choose the $\lambda$ coefficients from the interior of the set that resembles an ellipse with apsides in $0$ and $-\frac{\pi}{2\tau}$, and is elongated towards the latter one.

Referring now to Definition~\ref{defn finite- and infinite-time admissibility} and the mild solution of the $k$-th component \eqref{eqn k-th component abstract Cauchy problem mild solution} we introduce the forcing operator $\Phi_{\infty}\in\mathcal{L}(L^2(0,\infty;\CC),\mathcal{X}_{-1})$,

\begin{equation}\label{eqn defn forcing operator}
\Phi_{\infty}(u):=\int_{0}^{\infty}\mathcal{T}(t)\mathcal{B}u(t) \, dt,
\end{equation}
where
\[
\mathcal{T}(t)\mathcal{B}=\left(\begin{array}{cc} 
								\mathcal{T}_{11}(t) & \mathcal{T}_{12}(t) \\ 
								\mathcal{T}_{21}(t) & \mathcal{T}_{22}(t)	
						\end{array}\right)
						\left(\begin{array}{c} 
								b \\ 
								0	
						\end{array}\right)=
						\left(\begin{array}{c} 
								\mathcal{T}_{11}(t)b\\ 
								\mathcal{T}_{21}(t)b	
						\end{array}\right).
\]
Hence the forcing operator becomes
\begin{equation}\label{eqn k-th component forcing operator}
\Phi_{\infty}(u)=\left(\begin{array}{c} 
								\int_{0}^{\infty}\mathcal{T}_{11}(t)bu(t) \, dt\\ 
								\\
								\int_{0}^{\infty}\mathcal{T}_{21}(t)bu(t) \, dt	
						\end{array}\right)\in\mathcal{X}_{-1}.
\end{equation}

We can represent formally a similar product with the resolvent operator $R(s,\mathcal{A})$ from \eqref{eqn k-th component resolvent of abstract A}, namely
\begin{equation}\label{eqn k-th component resolvent times B}
R(s,\mathcal{A})\mathcal{B}=\left(\begin{array}{cc} 
									R_{11}(s) & R_{12}(s) \\ 
									R_{21}(s) & R_{22}(s)	
						\end{array}\right)
						\left(\begin{array}{c} 
								b \\ 
								0	
						\end{array}\right)=
						\frac{b}{s-\lambda \exp^{-s\tau}}
						\left(\begin{array}{c} 
								1\\ 
								\epsilon_s	
						\end{array}\right).					
\end{equation}
where the correspondence of sub-indices with elements of \eqref{eqn k-th component resolvent of abstract A} is obvious and will be used from now on to shorten the notation. 

The connection between the semigroup $\mathcal{T}(t)$ and the resolvent $R(s,\mathcal{A})$ is given by the Laplace transform (see e.g. \cite[Chapter 2.3]{Tucsnak_Weiss}) whenever the integral converges and 
\begin{equation}\label{eqn k-th component resolvent as Laplace transform}
R(s,\mathcal{A})\mathcal{B}=\int_{0}^{\infty}\exp^{-sr}\mathcal{T}(r)\mathcal{B} \, dr=b\left(\begin{array}{c} 
								\mathcal{L}(\mathcal{T}_{11})(s)\\ 
								\\
								\mathcal{L}(\mathcal{T}_{21})(s)	
						\end{array}\right)\in\mathcal{L}(\CC,\mathcal{X}_{-1}).
\end{equation}
We can now state the main theorem for the $k$-th component  of the delay system \eqref{eqn diagonal non-autonomous system}, namely
\begin{thm}\label{thm k-th component infty admissiblity}
Let for the given delay $\tau$ the eigenvalue $\lambda$ 
satisfy $\lambda \in \Lambda_\tau$.
Then the control operator $\mathcal{B}=\binom{b}{0}$ for the system \eqref{eqn k-th component abstract Cauchy problem} is infinite-time admissible for every $u\in L^2(0,\infty;\CC)$ and 
\[
\norm{\Phi_{\infty}u}_{\mathcal{X}}^2\leq\abs{b}^2(1+\tau)\frac{1}{\pi}\bigg(\frac{2-\delta}{\delta\abs{\lambda}}+\frac{2\delta}{(1-m^2)\abs{\lambda}}\bigg)\norm{u}_{L^2(0,\infty;\CC)}^2,
\]
for some $\delta,m\in(0,1)$, which can be given explicitly in terms of $\lambda$.
\end{thm}
\begin{proof}
\begin{itemize}
\item[1.] Consider the standard inner product on $L^2(0,\infty;\CC)$, namely 
\[
\inner{f,g}_{L^2(0,\infty;\CC)}=\int_{0}^{\infty}f(t)\bar{g}(t) \, dt\qquad\forall f,g\in L^2(0,\infty;\CC).
\]

Using \eqref{eqn k-th component forcing operator} and \eqref{eqn defn k-th component extended space} we may write
\begin{equation}\label{eqn forcing operator 1st component as inner product}
\int_{0}^{\infty}\mathcal{T}_{11}(t)bu(t) \, dt=b\inner{\mathcal{T}_{11},\bar{u}}_{L^2(0,\infty;\CC)}
\end{equation}
assuming that $\mathcal{T}_{11}\in L^2(0,\infty;\CC)$. This assumption is equivalent, due to the Paley--Wiener Theorem \ref{thm Paley-Wiener}, to $\mathcal{L}(\mathcal{T}_{11})\in H^2(\CC_+)$, where the last inclusion holds. Indeed, using \eqref{eqn k-th component resolvent times B} and \eqref{eqn k-th component resolvent as Laplace transform} we see that $\mathcal{L}(\mathcal{T}_{11})(s)=bR_{11}(s)=\frac{b}{s-\lambda \exp^{-s\tau}}$. Now the assumption on $\lambda$ gives $R_{11}\in H^2(\CC_+)$ and the result follows.

\item[2.] The boundary trace $R_{11}^*=\mathcal{L}(\mathcal{T}_{11})^*\in L^2(i\RR)$ is given a.e. as
\[
\mathcal{L}(\mathcal{T}_{11})^*(i\omega)=\frac{1}{i\omega-\lambda \exp^{-i\omega\tau}}.
\]
Again by Theorem \ref{thm Paley-Wiener} and definition of the inner product on $H^2(\CC)_+$ in \eqref{eqn defn inner product on H2 space} we have
\begin{align*}
&b\inner{\mathcal{T}_{11},\bar{u}}_{L^2(0,\infty;\CC)}=b\inner{\mathcal{L}(\mathcal{T}_{11})^*,\mathcal{L}(\bar{u})^*}_{L^2(i\RR)}\\
&=\frac{b}{2\pi}\int_{-\infty}^{+\infty}\frac{1}{i\omega-\lambda \exp^{-i\omega\tau}}\overbar{\mathcal{L}(\bar{u})}^*(i\omega) \, d\omega
\end{align*}
The Cauchy--Schwarz inequality now gives 
\begin{align*}
&\abs{b}\bigg\vert\frac{1}{2\pi}\int_{-\infty}^{+\infty}\frac{1}{i\omega-\lambda \exp^{-i\omega\tau}}\overbar{\mathcal{L}(\bar{u})}^*(i\omega) \, d\omega\bigg\vert\\
&\leq\abs{b}\bigg(\frac{1}{2\pi}\int_{-\infty}^{+\infty} \Big\vert\frac{1}{i\omega-\lambda \exp^{-i\omega\tau}}\Big\vert^{2} \, d\omega\bigg)^{\frac{1}{2}} \bigg(\frac{1}{2\pi}\int_{-\infty}^{+\infty} \big\vert \overbar{\mathcal{L}(\bar{u})}^*(i\omega)\big\vert^2 \, d\omega\bigg)^{\frac{1}{2}}\\
&=\abs{b}\bigg(\frac{1}{2\pi}\int_{-\infty}^{+\infty} \Big\vert\frac{1}{i\omega-\lambda \exp^{-i\omega\tau}}\Big\vert^{2} \, d\omega\bigg)^{\frac{1}{2}} \norm{u}_{L^2(0,\infty;\CC)},
\end{align*}
Combining this result with point 1 we obtain
\begin{equation}\label{eqn forcing operator 1st component norm}
\bigg\vert\int_{0}^{\infty}\mathcal{T}_{11}(t)bu(t) \, dt\bigg\vert^2\leq\abs{b}^2\bigg(\frac{1}{2\pi}\int_{-\infty}^{+\infty} \Big\vert\frac{1}{i\omega-\lambda \exp^{-i\omega\tau}}\Big\vert^{2} \, d\omega\bigg) \norm{u}_{L^2(0,\infty;\CC)}^2.
\end{equation}

\item[3.] Consider now the second element of the forcing operator \eqref{eqn k-th component forcing operator}, namely
\[
\int_{0}^{\infty}\mathcal{T}_{21}(t)bu(t) \, dt\in W^{-1,2}(-\tau,0;\CC).
\]
To shorten the notation we write $W:=W^{-1,2}(-\tau,0;\CC)$. If we assume that $\mathcal{T}_{21}\in  L^2(0,\infty;W)$ then using the vector-valued version of Theorem \ref{thm Paley-Wiener} this is equivalent to $\mathcal{L}(\mathcal{T}_{21})\in H^2(\CC_+,W)$, but the last inclusion holds. Indeed, to show it notice that 
\[
\epsilon_s(\sigma):=\exp^{s\sigma},\quad \sigma\in[-\tau,0]
\]
is, as a function of $s$, analytic everywhere for every value of  $\sigma$, and follow exactly the reasoning in point 1.

\item[4.] We introduce an auxiliary function $\phi:[0,\infty)\rightarrow\CC$. For that purpose fix $\mathcal{T}_{21}\in L^2(0,\infty;W)$ and $x_0\in W$ and define
\[
\phi(t):=\inner{\mathcal{T}_{21}(t),x_0}_W.
\] 
The Cauchy--Schwarz inequality gives 
\[
\int_{0}^{\infty}\abs{\inner{\mathcal{T}_{21}(t),x_0}_W}^{2} \, dt\leq\int_{0}^{\infty}\norm{\mathcal{T}_{21}(t)}_{W}^{2} \, dt\norm{x_0}_{W}^{2}<\infty,
\]
hence $\phi\in L^2(0,\infty;\CC)$.

\item[5.] Consider now the following:
\[
b\int_{0}^{\infty}\phi(t)u(t) \, dt =b\int_{0}^{\infty}\inner{\mathcal{T}_{21}(t),x_0}_{W}u(t)dt 
 =b\bigg\langle\int_{0}^{\infty}\mathcal{T}_{21}(t)u(t) \, dt,x_0\bigg\rangle_{W}.
\]
We also have 
\begin{align*}
b\int_{0}^{\infty}\phi(t)u(t) \, dt=b\inner{\phi,\bar{u}}_{L^2(0,\infty;\CC)}=b\inner{\mathcal{L}(\phi)^*,\mathcal{L}(\bar{u})^*}_{L^2(i\RR)}.
\end{align*}
To obtain the boundary trace $\mathcal{L}(\phi)^*$ notice that
\begin{align*}
\mathcal{L}(\phi)(s)&=\int_{0}^{\infty}\exp^{-sr}\inner{\mathcal{T}_{21}(r),x_0}_{W} \, dr=\bigg\langle\int_{0}^{\infty}\exp^{-sr}\mathcal{T}_{21}(r) \, dr,x_0\bigg\rangle_{W}\\
&=\inner{\mathcal{L}(\mathcal{T}_{21})(s),x_0}_{W}=\inner{R_{21}(s),x_0}_{W}.
\end{align*}
Using now \eqref{eqn k-th component resolvent times B} yields the result
\[
\mathcal{L}(\phi)^*(i\omega)=\inner{R_{21}^*(i\omega),x_0}_{W}=\bigg\langle\frac{\epsilon_{i\omega}}{i\omega-\lambda\exp^{-i\omega\tau}},x_0\bigg\rangle_{W}.
\]
Finally, using the inner product on $L^2(i\RR)$ and the fact that $\overbar{\mathcal{L}(\bar{u})}^*(i\omega)\in\CC$ for every $\omega\in\RR$ we obtain
\[
\bigg\langle\int_{0}^{\infty}\mathcal{T}_{21}(t)u(t) \, dt,x_0\bigg\rangle_{W}=\bigg\langle\frac{1}{2\pi}\int_{-\infty}^{+\infty}R_{21}^*(i\omega)\overbar{\mathcal{L}(\bar{u})}^*(i\omega) \, d\omega,x_0\bigg\rangle_{W}
\]
and 
\begin{equation}\label{eqn forcing operator 2nd component}
\int_{0}^{\infty}\mathcal{T}_{21}(t)u(t) \, dt
=\frac{1}{2\pi}\int_{-\infty}^{+\infty}R_{21}^*(i\omega)\overbar{\mathcal{L}(\bar{u})}^*(i\omega) \, d\omega\in W.
\end{equation}

\item [6.] Using the norm on $L^2(-\tau,0;\CC)$ we have
\begin{align*}
\norm{R_{21}^*(i\omega)}_{L^2(-\tau,0;\CC)}^{2}&=\int_{-\tau}^{0}\bigg\vert \frac{\exp^{i\omega t}}{i\omega-\lambda \exp^{-i\omega\tau}}\bigg\vert^{2} \, dt
= \frac{1}{\vert i\omega-\lambda \exp^{-i\omega\tau}\vert^{2}}\int_{-\tau}^{0}\big\vert\exp^{i\omega t}\big\vert^{2} \, dt\\
&=\frac{\tau}{\vert i\omega-\lambda \exp^{-i\omega\tau}\vert^{2}}.
\end{align*}

The Cauchy--Schwarz inequality gives
\begin{align*}
&\abs{b}\bigg\| \frac{1}{2\pi}\int_{-\infty}^{+\infty}R_{21}^*(i\omega)\overbar{\mathcal{L}(\bar{u})}^*(i\omega)
\, d\omega \bigg\|_{L^2(-\tau,0;\CC)}\\
&\leq\abs{b} \frac{1}{2\pi}\int_{-\infty}^{+\infty}\norm{R_{21}^*(i\omega)}_{L^2(-\tau,0;\CC)}\abs{\overbar{\mathcal{L}(\bar{u})}^*(i\omega)} \, d\omega\\
&=\abs{b} \frac{1}{2\pi}\int_{-\infty}^{+\infty}\frac{\tau^{\frac{1}{2}}}{\vert i\omega-\lambda \exp^{-i\omega\tau}\vert}\abs{\overbar{\mathcal{L}(\bar{u})}^*(i\omega)} \, d\omega\\
&\leq\abs{b}\bigg(\frac{1}{2\pi}\int_{-\infty}^{+\infty} \Big(\frac{\tau^{\frac{1}{2}}}{\abs{i\omega-\lambda \exp^{-i\omega\tau}}}\Big)^{2} \, d\omega\bigg)^{\frac{1}{2}} \bigg(\frac{1}{2\pi}\int_{-\infty}^{+\infty} \big\vert \overbar{\mathcal{L}(\bar{u})}^*(i\omega)\big\vert^2 \, d\omega\bigg)^{\frac{1}{2}}\\
&=\abs{b}\bigg(\frac{1}{2\pi}\int_{-\infty}^{+\infty} \frac{\tau}{\abs{i\omega-\lambda \exp^{-i\omega\tau}}^{2}}
\, d\omega\bigg)^{\frac{1}{2}} \norm{u}_{L^2(0,\infty;\CC)}
\end{align*}
Combining this result with point 5 gives 
\begin{equation}\label{eqn forcing operator 2nd component norm}
\bigg\|\int_{0}^{\infty}\mathcal{T}_{21}(t)bu(t) \, dt\bigg\|_{L^2(-\tau,0;\CC)}
\leq\abs{b}\bigg(\frac{1}{2\pi}\int_{-\infty}^{+\infty} \frac{\tau}{\abs{i\omega-\lambda \exp^{-i\omega\tau}}^{2}}
\, d\omega\bigg)^{\frac{1}{2}} \norm{u}_{L^2(0,\infty;\CC)}
\end{equation}

\item[7.] Taking now  the norm $\norm{\cdot}_{\mathcal{X}}$ resulting from \eqref{eqn defn k-th component inner product Cartesian product} and using \eqref{eqn k-th component forcing operator}, \eqref{eqn forcing operator 1st component norm} and \eqref{eqn forcing operator 2nd component norm} we arrive at
\begin{align}
\norm{\Phi_{\infty}(u)}_{\mathcal{X}}^{2}&=\bigg\vert\int_{0}^{\infty}\mathcal{T}_{11}(t)bu(t) \, dt\bigg\vert^2+\bigg\|\int_{0}^{\infty}\mathcal{T}_{21}(t)bu(t) \, dt\bigg\|_{L^2(-\tau,0;\CC)}^{2}\nonumber\\
&\leq\abs{b}^{2}(1+\tau)\bigg(\frac{1}{2\pi}\int_{-\infty}^{+\infty} \frac{1}{\abs{i\omega-\lambda \exp^{-i\omega\tau}}^{2}} \, d\omega\bigg) \norm{u}_{L^2(0,\infty;\CC)}^{2}\label{eqn forcing operator partial bound}
\end{align}
The remaining part is to deal with the integral in the above estimation. Note, that trying to calculate it directly this problem is equivalent (up to a constant) to calculation of the integral 
\[
 \int_{-i\infty}^{+i\infty}\,  \frac{ds}{(s-\lambda \exp^{-s\tau})(-s-\bar{\lambda} \exp^{s\tau})}
\]
with $s$ and $\lambda$ as complex variables. This inevitably leads to the Lambert-W function related pole placement and complications with finding a suitable contour of integration. To avoid these difficulties we will content ourselves with estimation only.

\item[8.]  Define 
\[
 \int_{-\infty}^{+\infty}I(i\omega) \, d\omega:=\int_{-\infty}^{+\infty} \frac{1}{\abs{i\omega-\lambda \exp^{-i\omega\tau}}^2} \, d\omega.
\]
From the fact that for $\omega\in\RR$ there is
\[
\abs{i\omega-\lambda\exp^{-i\omega\tau}}=\abs{-i\omega-\bar{\lambda}\exp^{i\omega\tau}},
\]
the equalities 
\begin{align*}
\int_{0}^{+\infty} \frac{1}{\abs{i\omega-\lambda \exp^{-i\omega\tau}}^2} \, d\omega&=\int_{-\infty}^{0} \frac{1}{\abs{i\omega-\bar{\lambda} \exp^{-i\omega\tau}}^2} \, d\omega,\\
\int_{-\infty}^{0} \frac{1}{\abs{i\omega-\lambda \exp^{-i\omega\tau}}^2} \, d\omega&=\int_{0}^{+\infty} \frac{1}{\abs{i\omega-\bar{\lambda} \exp^{-i\omega\tau}}^2} \, d\omega
\end{align*}
follow. They give 
\[
\int_{-\infty}^{+\infty} \frac{1}{\abs{i\omega-\lambda \exp^{-i\omega\tau}}^2} \, d\omega=\int_{-\infty}^{+\infty} \frac{1}{\abs{i\omega-\bar{\lambda} \exp^{-i\omega\tau}}^2} \, d\omega,
\]
and therefore one can consider the sole case $\Arg(\lambda)\in(\frac{\pi}{2},\pi]$.
Using the reverse triangle inequality we may now write 
\begin{equation}\label{eqn k-th component norm integral decomposition}
\begin{split}
\int_{-\infty}^{+\infty}I(i\omega) \, d\omega &=
\int_{-\infty}^{0} \frac{1}{\abs{i\omega-\lambda \exp^{-i\omega\tau}}^2} \, d\omega+
\int_{0}^{+\infty} \frac{1}{\abs{i\omega-\lambda \exp^{-i\omega\tau}}^2} \, d\omega \\
      &=\int_{0}^{+\infty} \frac{1}{\abs{i\omega-\bar{\lambda} \exp^{-i\omega\tau}}^2} \, d\omega+
           \int_{0}^{+\infty} \frac{1}{\abs{i\omega-\lambda \exp^{-i\omega\tau}}^2} \, d\omega\\
  &\leq\int_{0}^{(1-\delta)\abs{\bar{\lambda}}} \, \frac{d\omega}{\big(\abs{\omega}-\abs{\bar{\lambda}}\big)^2}+
           \int_{(1+\delta)\abs{\bar{\lambda}}}^{+\infty}\, \frac{d\omega}{\big(\abs{\omega}-\abs{\bar{\lambda}}\big)^2}\\
      &+\int_{(1-\delta)\abs{\bar{\lambda}}}^{(1+\delta)\abs{\bar{\lambda}}}\frac{1}{\abs{i\omega-\bar{\lambda} \exp^{-i\omega\tau}}^2} \, d\omega\\
      &+\int_{0}^{(1-\delta)\abs{\lambda}} \, \frac{d\omega}{\big(\abs{\omega}-\abs{\lambda}\big)^2}+
           \int_{(1+\delta)\abs{\lambda}}^{+\infty}\, \frac{d\omega}{\big(\abs{\omega}-\abs{\lambda}\big)^2}\\
      &+\int_{(1-\delta)\abs{\lambda}}^{(1+\delta)\abs{\lambda}} \frac{1}{\abs{i\omega-\lambda \exp^{-i\omega\tau}}^2} \, d\omega\\
    &=2\int_{0}^{(1-\delta)\abs{\lambda}} \, \frac{d\omega}{\big(\omega-\abs{\lambda}\big)^2}+
         2\int_{(1+\delta)\abs{\lambda}}^{+\infty}\, \frac{d\omega}{\big(\omega-\abs{\lambda}\big)^2}\\
      &+\int_{(1-\delta)\abs{\bar{\lambda}}}^{(1+\delta)\abs{\bar{\lambda}}}\frac{1}{\abs{i\omega-\bar{\lambda} \exp^{-i\omega\tau}}^2} \, d\omega+
           \int_{(1-\delta)\abs{\lambda}}^{(1+\delta)\abs{\lambda}} \frac{1}{\abs{i\omega-\lambda \exp^{-i\omega\tau}}^2} \, d\omega
\end{split}
\end{equation}
for any $\delta\in(0,1)$. 
The first and second integral on the right hand side give
\[
 \int_{0}^{(1-\delta)\abs{\lambda}} \, \frac{d\omega}{\big(\omega-\abs{\lambda}\big)^2}=\frac{1-\delta}{\delta\abs{\lambda}},
\]
and
\[
 \int_{(1+\delta)\abs{\lambda}}^{+\infty} \, \frac{d\omega}{\big(\omega-\abs{\lambda}\big)^2}=\frac{1}{\delta\abs{\lambda}}.
\]
Taking into account the comments above we will firstly find the upper bound for the last integral in \eqref{eqn k-th component norm integral decomposition}.

\item[9.]
Hence, using the assumption let $\lambda=\abs{\lambda}\exp^{i\Arg(\lambda)}$, where $\Arg(\lambda)=\frac{\pi}{2}+\varepsilon_\lambda$, $\varepsilon_\lambda\in(0,\frac{\pi}{2}]$ and
\[
 0<\abs{\lambda}\tau<\Arg(\lambda)-\frac{\pi}{2}=\varepsilon_\lambda\leq\frac{\pi}{2}.
\]
Fix now $\delta\in(0,1)$ such that 
\begin{equation}\label{eqn delta condition}
(1-\delta)\abs{\lambda}\tau<\abs{\lambda}\tau<(1+\delta)\abs{\lambda}\tau<\varepsilon_\lambda.
\end{equation}
Let $\eta\in[1-\delta,1+\delta]$ and consider $\omega=\eta\abs{\lambda}$. For such $\omega$ we have
\begin{equation}\label{eqn denominator of I(iw)}
\begin{split}
 \big\vert i\omega-\lambda \exp^{-i\omega\tau}\big\vert^2&=\abs{\lambda}^2\big\vert\eta\exp^{i\frac{\pi}{2}}-\exp^{i(\Arg(\lambda)-\eta\abs{\lambda}\tau)}\big\vert^2\\
 &=\abs{\lambda}^2\abs{k_\eta-q_\eta}^2=\abs{\lambda}^2v_{\eta}^2
\end{split}
\end{equation}
with the obvious definition of $k_\eta$ and $q_\eta$ vectors and $v_\eta:=\abs{k_\eta-q_\eta}$. Due to \eqref{eqn delta condition} and the definition of $\eta$ we have
\[
 \Arg(\lambda)-\eta\abs{\lambda}\tau=\frac{\pi}{2}+\varepsilon_\lambda-\eta\abs{\lambda}\tau>\frac{\pi}{2}
\]
and $\abs{k_\eta-q_\eta}^2>0$ for every $\eta\in[1-\delta,1+\delta]$. Define $\varepsilon(\eta)$ as the angle between $k_\eta$ and $q_\eta$, that is
\[
 \varepsilon(\eta):=\varepsilon_\lambda-\eta\abs{\lambda}\tau,
\]
which is a linear function of $\eta\in[1-\delta,1+\delta]$ with values 
\begin{equation}\label{eqn angle between vectors on the 1-circle}
\varepsilon(\eta)\in\big(\varepsilon_\lambda-(1+\delta)\abs{\lambda}\tau,\varepsilon_\lambda-(1-\delta)\abs{\lambda}\tau\big)\subset (0,\frac{\pi}{2}).
\end{equation}

The law of cosines in the $\eta$-dependent triangle $(k_\eta,q_\eta,v_\eta)$ gives
\begin{equation}\label{eqn difference of vectors on the 1-circle}
  v_{\eta}^2=1+\eta^2-2\eta\cos\big(\varepsilon_\lambda-\eta\abs{\lambda}\tau\big).
\end{equation}
The strict monotonicity of the cosine function on $(0,\frac{\pi}{2})$ and \eqref{eqn angle between vectors on the 1-circle} give
\begin{equation}\label{eqn defn max of cosine}
m:=\max\big\{\cos\big(\varepsilon_\lambda-\eta\abs{\lambda}\tau\big):\eta\in[1-\delta,1+\delta]\big\}=\cos\big(\varepsilon_\lambda-(1+\delta)\abs{\lambda}\tau\big)
\end{equation}

and $m\in(0,1)$. Hence, for every $\eta\in[1-\delta,1+\delta]$ we have
\begin{equation}\label{eqn difference lower bound}
v_{\eta}^2\geq 1+\eta^2-2\eta m\geq 1-m^2>0,
\end{equation}

as $\eta^2-2\eta m+m^2=(\eta-m)^2\geq0$ so that $\eta^2-2\eta m\geq -m^2$.

Now \eqref{eqn denominator of I(iw)} and \eqref{eqn difference lower bound} give
\[
\frac{1}{\big\vert i\omega-\lambda \exp^{-i\omega\tau}\big\vert^2}\leq\frac{1}{(1-m^2)\abs{\lambda}^2}\qquad\forall\omega\in\big[(1-\delta)\abs{\lambda},(1+\delta)\abs{\lambda}\big]
\]
and, in consequence, lead to a finite upper bound of the last integral in \eqref{eqn k-th component norm integral decomposition}, that is
\begin{equation}\label{eqn  k-th component norm integral decomposition 2nd bound}
\int_{(1-\delta)\abs{\lambda}}^{(1+\delta)\abs{\lambda}} \frac{1}{\abs{i\omega-\lambda \exp^{-i\omega\tau}}^2} \, d\omega
\leq\frac{2\delta}{(1-m^2)\abs{\lambda}}.
\end{equation}
Noting that $\Arg(\bar{\lambda})=-\Arg(\lambda)$ and using the same geometrical approach one can show that for the third integral in \eqref{eqn k-th component norm integral decomposition} the upper bound of  \eqref{eqn  k-th component norm integral decomposition 2nd bound} also holds.

\item[10.] Taking together \eqref{eqn forcing operator partial bound}, \eqref{eqn k-th component norm integral decomposition} and \eqref{eqn  k-th component norm integral decomposition 2nd bound} we arrive at
\[
\norm{\Phi_{\infty}(u)}_{\mathcal{X}}^{2}\leq\abs{b}^{2}(1+\tau)\frac{1}{\pi}\bigg(\frac{2-\delta}{\delta\abs{\lambda}}+\frac{2\delta}{(1-m^2)\abs{\lambda}}\bigg)\norm{u}_{L^2(0,\infty;\CC)}^{2}
\]

\end{itemize}
\end{proof}

\subsection{Analysis of the whole system}

Let us return now to the diagonal non-autonomous system \eqref{eqn diagonal non-autonomous system} with state space $X=l^2(\CC)$ and to denoting its $k$-th component with the subscript. As shown in the previous subsection, Proposition~\ref{prop wellposedness of k-th component ACP}  states that the system \eqref{eqn k-th component of diagonal non-autonomous system} describing the $k$-th component is well-posed and its mild solution is given by \eqref{eqn k-th component abstract Cauchy problem mild solution}, that is  $v_{k}:[0,\infty)\rightarrow\mathcal{X}$,

\begin{equation}\label{eqn k-th component abstract Cauchy problem mild solution with indices}
v_{k}(t)=\binom{z_{k}(t)}{z_{t_k}}=\mathcal{T}_{k}(t)v_{k}(0)+\int_{0}^{t}\mathcal{T}_{k}(t-s)\mathcal{B}_{k}u(s) \, ds.
\end{equation}
Given the structure of the Hilbert space $\mathcal{X}$ in \eqref{eqn defn of inner product on Cartesian product} the mild solution \eqref{eqn k-th component abstract Cauchy problem mild solution with indices} has values in the subspace of $\mathcal{X}$ spanned by the $k$-th element of its basis. Hence, defining $v:[0,\infty)\rightarrow\mathcal{X}$,

\begin{equation}\label{eqn abstract Cauchy problem mild solution with indices}
v(t):=\sum_{k\in\NN}v_{k}(t),
\end{equation}
we obtain a unique mild solution of \eqref{eqn diagonal non-autonomous system} and this system is well-posed. 
Using \eqref{eqn abstract Cauchy problem mild solution with indices} and \eqref{eqn defn of inner product on Cartesian product} we have 
\begin{equation}\label{eqn abstract Cauchy problem mild solution norm in X}
\begin{split}
\norm{v(t)}_{\mathcal{X}}^2&=\bigg\|\binom{z(t)}{z_{t}}\bigg\|_{\mathcal{X}}^2=\norm{z(t)}_{l^2}^2+\norm{z_t}_{L^2(-\tau,0;l^2)}^2\\
			   &=\sum_{k\in\NN}\abs{z_k(t)}^2+\sum_{k\in\NN}\abs{\inner{z_t,l_k}_{L^2(-\tau,0;l^2)}}^2\\
			   &=\sum_{k\in\NN}\bigg(\abs{z_k(t)}^2+\norm{z_{t_k}}_{L^2(-\tau,0;\CC)}^2\bigg)\\
			   &=\sum_{k\in\NN}\norm{v_k(t)}_{\mathcal{X}}^2,
\end{split}
\end{equation}
where we used again \eqref{eqn defn k-th component Cartesian product} and notation from \eqref{eqn k-th component of diagonal non-autonomous system}. We can formally write the mild solution \eqref{eqn abstract Cauchy problem mild solution with indices} as a function $v:[0,\infty)\rightarrow\mathcal{X}_{-1}$,

\begin{equation}\label{eqn whole abstract Cauchy problem mild solution formal}
v(t)=\mathcal{T}(t)v(0)+\int_{0}^{t}\mathcal{T}(t-s)\mathcal{B}u(s) \, ds.
\end{equation}
where $\mathcal{X}_{-1}=X_{-1}\times W^{-1,2}(-\tau,0;X)$ and  the control operator $\mathcal{B}\in\mathcal{L}(\CC,\mathcal{X}_{-1})$ is given by  $\mathcal{B}=\binom{(b_k)_{k\in\NN}}{0}$. We may now state the main theorem of this article.

\begin{thm}\label{thm whole system infty admissibility}
 Let for the given delay $\tau$ every element of the sequence $(\lambda_k)_{k\in\NN}$ satisfy $\lambda_k \in \Lambda_\tau$, where $\Lambda_\tau$ was defined in \eqref{eq:lambdatau}.
Then the control operator $\mathcal{B}\in\mathcal{L}(\CC,\mathcal{X}_{-1})$ given by $\mathcal{B}=\binom{(b_k)_{k\in\NN}}{0}$ is infinite-time admissible   if the sequence $(C_k)_{k\in\NN}\in l^1$, where
\[
 C_k:=\abs{b_k}^2(1+\tau)\frac{1}{\pi}\bigg(\frac{2-\delta_k}{\delta_k\abs{\lambda_k}}+\frac{2\delta_k}{(1-m_k^2)\abs{\lambda_k}}\bigg)
\]
and $\delta_k,m_k$ fulfil the conditions \eqref{eqn delta condition} and \eqref{eqn defn max of cosine}.

\end{thm}

\begin{proof}
 Define the forcing operator for \eqref{eqn whole abstract Cauchy problem mild solution formal} as $\Phi_\infty:L^2(0,\infty)\rightarrow\mathcal{X}_{-1}$,
\[
 \Phi_{\infty}(u):=\int_{0}^{\infty}\mathcal{T}(t)\mathcal{B}u(t) \, dt.
\]
From \eqref{eqn abstract Cauchy problem mild solution with indices} it can be represented as
\[
 \Phi_{\infty}(u)=\sum_{k\in\NN}\Phi_{\infty_k}(u),
\]
where $\Phi_{\infty_k}(u)$ is given by \eqref{eqn defn forcing operator} for every $k\in\NN$. Then, similarly as in \eqref{eqn abstract Cauchy problem mild solution norm in X} and using the assumption we see that 
\[
 \norm{\Phi_{\infty}(u)}_{\mathcal{X}}^2=\sum_{k\in\NN}\norm{\Phi_{\infty_k}(u)}_{\mathcal{X}}^2\leq\bigg(\sum_{k\in\NN}\abs{C_k}\bigg)\norm{u}_{L^2(0,\infty;\CC)}^2<\infty
\]
\end{proof}
\section{Examples}\label{sec:4}

In construction of an appropriate example fulfilling assumptions of Theorem~\ref{thm whole system infty admissibility} the biggest difficulty lies in the condition imposed on the eigenvalues $(\lambda_k)_{k\in\NN}$ of the generator $(A,D(A))$, defined  in \eqref{eqn defn diagonal generator}. Apart from a somewhat artificial case where one could simply define $\lambda_k:=(-\frac{\pi}{2\tau}+\varepsilon)\frac{1}{k}$ for some fixed $\tau,\varepsilon>0$ and all $k\in\NN$, we provide two additional, more illustrative examples.

\subsection{Multiplication operator}\label{Example: multiplication operator}
Consider the multiplication operator on the space $L^2(\Omega,\mu)$, with a $\sigma$-finite measure space $(\Omega,\mathfrak{M},\mu)$, as shown and described in detail in \cite[Section I.4.b]{Engel_Nagel}. More precisely, for a measurable function  (called \textit{symbol}) $q:\Omega\rightarrow\CC$, we call the set 
\[
 q_{ess}(\Omega):=\Big\{\lambda\in\CC:\mu\big(\{s\in\Omega:\abs{q(s)-\lambda}<\varepsilon\}\big)\neq0\hbox{ for all }\varepsilon>0\Big\}
\]
the \textit{essential range} of $q$ and define the associated multiplication operator $M_q$ as 
\[
 M_qf:=q\cdot f,\quad D(M_q):=\big\{f\in L^2(\Omega,\mu):q\cdot f\in L^2(\Omega,\mu)\big\}.
\]
The importance of this example lies in the fact that each normal operator on a Hilbert space is unitarily equivalent to a multiplication operator on some $L^2$ space.

From the perspective of Theorem~\ref{thm whole system infty admissibility} the multiplication operator has a useful property, namely the spectrum of $M_q$ is the essential range of $q$, that is 
\[
\sigma(M_q)=q_{ess}(\Omega). 
\]
Hence, by choosing a suitable symbol it would be easy to control the eigenvalues. However, due to the boundedness of the region of interest in Theorem~\ref{thm whole system infty admissibility}, the symbol $q$ would have to be essentailly bounded, what is a neccessary and sufficient condition for the boundedness of the multiplication operator $M_q$ and would limit further considerations to uniformly bounded semigroups.

\subsection{Reciprocal system}\label{Example: reciprocal system}
Following \cite{Curtain_2003} we introduce the notion of a \textit{state linear system} $\Sigma(A,B,C,D)$ considered on the extrapolation space $X_{-1}$, where $B\in\mathcal{L}(U,X_{-1})$, $C\in\mathcal{L}(X_{1},Y)$, $D\in\mathcal{L}(U,Y)$, $A$ generates the semigroup $(T(t))$  on $X_{-1}$ and $G$ is a transfer function of this system. 

Suppose that system $\Sigma(A,B,C,D)$ is such that $0\in\rho(A)$. Then its \textit{reciprocal system} is the state linear system $\Sigma(A^{-1},A^{-1}B,-CA^{-1},G(0))$. This means that for a diagonal generator $A$ with eigenvalues $(\lambda_k)_{k\in\NN}$ the generator $A^{-1}$ of the reciprocal system has eigenvalues $(\lambda_k^{-1})_{k\in\NN}$.

Note that by Theorem 5 of \cite{Curtain_2003}, the operator $B$ is admissible for the semigroup $(T(t))$
if and only if $A^{-1}B$ is admissible for the reciprocal semigroup generated by $A^{-1}$.

Consider a heat propagation model in a homogeneous rod with zero temperature imposed on its both ends (see \cite[Example 2.6.9]{Tucsnak_Weiss} for more details). In terms of PDEs this model takes the form
\begin{equation}\label{eqn heat in rod by PDEs}
\left\{\begin{array}{lr}
        \frac{\partial w}{\partial t}(x,t)=\frac{\partial^2 w}{\partial x^2}(x,t),& x\in(0,\pi),t\geq0,\\
        w(0,t)=0,\ w(\pi,t)=0,& t\in[0,\infty),\\
        w(x,0)=w_0(x), & x\in(0,\pi),\\
        \end{array}
\right.
\end{equation}
where the temperature profile belongs to the state space $X=L^2(0,\pi)$, the initial condition (the initial temperature distribution) is $w_0\in W^{2,2}(0,\pi)\cap W_{0}^{1,2}(0,\pi)$.

Define
\[
Az:=\frac{d^2z}{dx^2},\quad D(A):= W^{2,2}(0,\pi)\cap W_{0}^{1,2}(0,\pi),
\]
and reformulate \eqref{eqn heat in rod by PDEs} into an abstract setting
\begin{equation}\label{eqn heat in rod abstract}
\dot{z}(t)=Az(t),\quad z(0)=w_0.
\end{equation}
Note also that $0\in\rho(A)$. For $k\in\NN$ let $\phi_k\in D(A)$, $\phi_k(x):=\sqrt{\frac{2}{\pi}}\sin(kx)$ for every $x\in(0,\pi)$. Then $(\phi_k)_{k\in\NN}$ is an orthornormal Riesz basis in $X$ and 
\[
 A\phi_k=-k^2\phi_k \qquad \forall k\in\NN.
\]
Using standard Hilbert space methods and transforming system \eqref{eqn heat in rod abstract} into the $l^2$ space (we use the same notation for the $l^2$ version of \eqref{eqn heat in rod abstract}) we see that the associated eigenvalue sequence $(\lambda_k)_{k\in\NN}$ is $\lambda_k=-k^2$ for every $k\in\NN$.

Take the delay $\tau=1$. Then the system being the reciprocal of \eqref{eqn heat in rod abstract} has a generator with a sequence of eigenvalues $(-\frac{1}{k^2})_{k\in\NN}$ fulfilling the assumption of Theorem~\ref{thm whole system infty admissibility}.

\section{Conclusions}\label{sec:5}

We have cast our results in the language of infinite-time admissibility, since this allowed us 
to make use of Laplace transform techniques, but since for exponentially stable
systems (with $\sup \re \lambda_k < 0 $)
this is equivalent to finite-time admissibility,   similar conclusions hold in this situation
as in our main theorem, Theorem \ref{thm whole system infty admissibility}.

The region $\Lambda_\tau$ is a very natural one to find in our analysis, as may be seen by observing
that the system with transfer function $1/(s+\lambda e^{-s\tau})$ (where $\tau>0$ and
$\lambda \in \CC$) is $H^\infty$ stable if and only if $\lambda \in \Lambda_\tau$.
Thus, paradoxically, a large negative eigenvalue $\lambda$, although seemingly contributing to
stability, actually causes destabilization, and loss of admissibility, in the presence of delays. 
Thus for a system such as the heat equation, where the set of eigenvalues is not contained in any
single $\Lambda_\tau$, one cannot expect a positive result in the presence of delay. 

This is also interesting from the reciprocal systems point of view, as given in Example \ref{Example: reciprocal system}, for the following reason. According to \cite[Theorem 5]{Curtain_2003}, $B$ is an infinite-time admissible operator if and only if $A^{-1}B$ is. As our analysis shows, adding a positive delay breaks this symmetry.

The last conclusion concerns the open question we formulated in \cite{Partington_Zawiski_2018a}, where we looked for admissibility criteria of retarded delay systems formed by contraction semigroups. In light of our results for diagonal state-delayed system it seems that contraction is not a sufficient condition for admissibility of a diagonal retarded delay system. Instead, sufficiency is reached when the sequence of eigenvalues of the undelayed semigroup fulfils a condition similar to $\lambda_k\in\Lambda_\tau$.

\subsection*{Acknowledgements}
This project has received funding from the European Union's Horizon 2020 research and innovation programme under the Marie Sk{\l}odowska-Curie grant agreement No 700833.\\



\bibliography{delay_systems_library}

\bibliographystyle{amsplain}


\end{document}